\documentclass[10pt,reqno]{amsart}
\usepackage{amsmath}
\usepackage{amsthm}
\usepackage{amssymb}
\usepackage{amsfonts}
\usepackage{latexsym}
\usepackage[hypertexnames=false]{hyperref}
\usepackage{cite}
\usepackage{xcolor}
\usepackage{enumitem}

\usepackage{mathtools}
\mathtoolsset{showonlyrefs}

\usepackage[font=footnotesize,labelfont=bf]{caption}

\usepackage{tikz}
\usetikzlibrary{decorations.markings}

\theoremstyle{plain}
\newtheorem{theorem}{Theorem}[section]
\newtheorem{lemma}[theorem]{Lemma}

\theoremstyle{definition}
\newtheorem{remark}[theorem]{Remark}

\newcommand{\Q}{\mathbb{Q}}

\newcommand{\K}{\mathbb{K}}

\renewcommand{\L}{\mathbb{L}}
\newcommand{\opn}[1]{\operatorname{#1}}
\newcommand{\Gal}{\opn{Gal}}
\newcommand{\Ind}{\opn{Ind}}
\newcommand{\triv}{\mathbf 1}
\newcommand{\dd}{\,\mathrm{d}}

\renewcommand{\Re}{\opn{Re}}

\allowdisplaybreaks

\begin{document}
\title[Explicit conditional bounds]{Explicit conditional bounds for the residue of a Dedekind zeta-function at $s=1$}

\author[S.~R.~Garcia]{Stephan Ramon Garcia}
\address[S.~R.~Garcia]{Department of Mathematics and Statistics\\
         Pomona College\\
         610 N. College Ave.\\
         Claremont, CA 91711\\
         USA}
\email{stephan.garcia@pomona.edu}
\urladdr{\url{https://stephangarcia.sites.pomona.edu/}}

\author[L.~Greni\'{e}]{Lo\"{\i}c Greni\'{e}}
\address[L.~Greni\'{e}]{Dipartimento di Ingegneria gestionale, dell'informazione e della produzione\\
         Universit\`{a} di Bergamo\\
         viale Marconi 5\\
         I-24044 Dal\-mi\-ne\\
         Italy}
\email{loic.grenie@gmail.com}

\author[E.~S.~Lee]{Ethan Simpson Lee}
\address[E.~S.~Lee]{University of the West of England\\
         School of Computing and Creative Technologies\\
         Coldharbour Lane\\
         Bristol BS16 1QY\\
         United Kingdom}
\email{ethan.lee@uwe.ac.uk}
\urladdr{\url{https://sites.google.com/view/ethansleemath/home}}

\author[G.~Molteni]{Giuseppe Molteni}
\address[G.~Molteni]{Dipartimento di Matematica\\
         Universit\`{a} di Milano\\
         via Saldini 50\\
         I-20133 Milano\\
         Italy}
\email{giuseppe.molteni1@unimi.it}
\urladdr{\url{https://sites.unimi.it/molteni/}}

\begin{abstract}
We prove new explicit conditional bounds for the residue at $s=1$ of the Dedekind zeta-function
associated to a number field. Our bounds are concrete and all constants are presented with explicit
numerical values.
\end{abstract}


\maketitle

\begin{center}
To appear in J. Number Theory (2026).\\
Electronically published on ?\\
DOI: \url{https://doi.org/10.1016/j.jnt.2026.08.010}\\
ARXIV: \url{https://arxiv.org/abs/2506.17416}
\end{center}

\section{Introduction}
Suppose that $\K\neq\Q$ is a number field with degree $n_{\K}$ and discriminant $\Delta_{\K}$. %
Let $\zeta_{\K}(s)$ denote the Dedekind zeta-function associated to $\K$ and let $\zeta = \zeta_{\Q}$
denote the Riemann zeta-function. Recall that the Generalized Riemann Hypothesis (GRH) for $\zeta_{\K}$
postulates that if $\zeta_{\K}(\sigma + it) = 0$ and $0 < \sigma < 1$, then $\sigma = 1/2$ and $t\neq 0$.

Since the residue $\kappa_{\K}$ of $\zeta_{\K}(s)$ at $s=1$ encodes important arithmetic information
about $\K$ via the class number formula, the bounds for this invariant are naturally interesting.
Conditionally, we can prove explicit bounds for $\kappa_{\K}$ of the shape
\begin{equation}\label{eqn:right_shape}
\frac{c_-}{\ln\ln|\Delta_{\K}|}
\leq \kappa_{\K}
\leq c_+(\ln\ln|\Delta_{\K}|)^{n_{\K}-1}
\end{equation}
for suitable functions $c_{\pm}$ that weakly depend on $n_{\K}$ or $|\Delta_{\K}|$. Cho and Kim
\cite{ChoKim} have proved, as a byproduct of their approach, that if GRH for $\zeta_{\K}$ and the
strong Artin conjecture for the quotient $\zeta_\K/\zeta$ are true, then
%
%
\begin{equation}\label{eq:ChoKim}
c_{-} = (1 + o(1))\frac{\zeta(n_{\K})}{2 e^{\gamma}}
\quad\text{and}\quad
c_{+} = (2e^{\gamma} + o(1))^{n_{\K}-1}
\end{equation}
are admissible in \eqref{eqn:right_shape}, where $\gamma = 0.57721\dots$ denotes the Euler--Mascheroni
constant. However, these authors did not specify the $o(1)$ terms, nor is it straightforward to extract
a simple computation from their arguments.

We prove the following theorem, which gives new bounds for $\kappa_{\K}$ that are explicit and match the
asymptotic strength of the bounds from \eqref{eq:ChoKim}. Although the constants $19$ in
\eqref{eq:BoundUpper} and \eqref{eq:BoundLower_ii} can be refined by examining various computations
within the body of this paper, we prefer the simpler expressions in Theorem \ref{Theorem:Main}.

\begin{theorem}\label{Theorem:Main}
Suppose that $\K$ is a number field such that $|\Delta_{\K}|\geq 14$, the GRH for $\zeta_{\K}$ is true,
and $\zeta_{\K}/\zeta$ is entire. Then
\begin{align}
\kappa_{\K}
 &\leq \Big(2 e^{\gamma}
      (\ln\ln{|\Delta_{\K}|})^{1 + \frac{19}{\ln\ln{|\Delta_{\K}|}}}\Big)^{n_{\K} - 1}
    \label{eq:BoundUpper}
\intertext{and}
\kappa_{\K}
 &\geq \frac{\zeta(n_{\K})}
            {2e^\gamma (\ln\ln{|\Delta_{\K}|})^{1 + \frac{19 (n_{\K} - 1)}{\ln\ln{|\Delta_{\K}|}}}} .
    \label{eq:BoundLower_ii}
\end{align}
\end{theorem}
The assumption $|\Delta_{\K}|\geq 14$ is necessary in order to have the upper bound function in
\eqref{eq:BoundUpper} be truly larger than the lower bound function in \eqref{eq:BoundLower_ii}.
Observe also that there are only $\Q$ and nine quadratic fields with $|\Delta_\K|\leq 13$.

Notice that in Theorem~\ref{Theorem:Main} we only assume that the quotient $\zeta_\K/\zeta$ is entire,
i.e. the usual Artin's hypothesis for this function, not the \emph{strong} Artin's conjecture (i.e. the
claim that the Artin $L$ function $\zeta_\K/\zeta$ is entire because it coincides with the $L$-function
coming from a suitable authomorphic representation). From this point of view, Theorem~\ref{Theorem:Main}
is much less demanding than what is proved in~\cite{ChoKim}.

To prove Theorem \ref{Theorem:Main}, we prove auxiliary bounds in Section \ref{Section:AuxiliaryBounds},
approximate $\ln{\kappa_{\K}}$ in Section \ref{Section:Duck}, and then complete the proof in Section
\ref{Section:Proof}. There are several results along the way that may be of independent interest.

There are some similarities between our arguments to prove Theorem \ref{Theorem:Main} and the arguments
employed by Cho and Kim \cite{ChoKim} to prove \eqref{eq:ChoKim}. However, our approaches differ in two
major ways. First, every argument in our work needs to be explicit, resulting in a large number of
intermediate steps and computations that must be carefully tracked. Second, the arguments we use to
approximate $\ln{\kappa_{\K}}$ in Section \ref{Section:Duck} refine the arguments employed in
\cite{ChoKim}. Specifically, our argument implements a smoothing technique to ensure stronger
computational outputs.

\begin{remark}
In earlier work, the first and third authors \cite{GarciaLeeArtinL} proved that if the GRH for
$\zeta_{\K}$ is true and $\zeta_{\K}/\zeta$ is entire, then one can take
\begin{equation*}
c_{-} = e^{- 17.81(n_{\K} - 1)}
\quad\text{and}\quad
c_{+} = e^{  17.81(n_{\K} - 1)}
\end{equation*}
in \eqref{eqn:right_shape}. Under the same assumptions, Paloj\"{a}rvi and Simoni\v{c} \cite[Cor.~4]{PalSim}
proved that if $|\Delta_{\K}| \geq 5.4\cdot 10^{6}$, then a stronger upper bound with
\begin{equation}\label{eq:PalSim}
c_{+} = \big(2 e^{\gamma + \frac{2.475}{\ln\ln{|\Delta_{\K}|}}}\big)^{n_{\K}-1}
\end{equation}
is available. We notice that our new lower bound \eqref{eq:BoundLower_ii} considerably improves on
{\color{red}these} previous results, while our new upper bound \eqref{eq:BoundUpper} is actually weaker
than \eqref{eq:PalSim}. This difference can be explained by different methodologies.
\end{remark}
\begin{remark}
In related work, Bessassi \cite[Thm.~10]{Bessassi} proved that if $\K$ is normal and $\zeta_\K$ satisfies
GRH, then
\begin{equation}\label{eqn:Bessassi}
c_{-} = \frac{1}{2e^\gamma} + o(1),
\end{equation}
in which the little-$o$ term is not made explicit but is effective. Upon examining their proof, one can
see that it is essentially of order $\ln\ln\ln|\Delta_\K|/\ln\ln|\Delta_\K|$. Although the lower bound of
\eqref{eqn:Bessassi} is worse than the lower bound of \eqref{eq:ChoKim} and \eqref{eq:BoundLower_ii} by
the factor $\zeta(n_\K)$, their result follows from different hypotheses.
\end{remark}
\begin{remark}
Louboutin \cite{Louboutin00, Louboutin01, Louboutin03, Louboutin05, Louboutin11} and the first and third
authors \cite{GarciaLeeMertens} proved several unconditional explicit bounds for $\kappa_{\K}$. For
example, these bounds tell us that if $\K\neq\Q$, then
\begin{equation}\label{eq:UncondKappa}
\frac{0.36232}{\sqrt{|\Delta_{\K}|}}
\leq \kappa_{\K}
\leq \Big(\frac{e\ln |\Delta_{\K}| }{2(n_{\K} - 1)}\Big)^{n_{\K} - 1} .
\end{equation}
Specifically, this upper bound is due to Louboutin \cite[Thm.~1]{Louboutin00} and this lower bound is due
to the first and third authors \cite[Sec.~3]{GarciaLeeMertens}. A further lower bound by Oesterl\'{e} is
given in the Appendix of \cite{Bessassi}. Currently, no method is known that could produce bounds of the
form \eqref{eqn:right_shape} without assuming the GRH for $\zeta_{\K}$.
\end{remark}

\subsection*{Acknowledgements}
LG and GM are members of the INdAM group GNSAGA. SRG is supported by NSF Grants DMS-2054002 and
DMS-2452084. ESL thanks the Heilbronn Institute for Mathematical Research for their support, as well as
Matt Bissatt, Ross Paterson, and other colleagues for helpful conversations concerning the properties of
Artin $L$-functions. We also thank Andy Booker, Aleksander Simoni\v{c}, and Tim Trudgian for valuable
feedback and the referee for their careful reading of this preprint and suggestions.

\section{Auxiliary results}\label{Section:AuxiliaryBounds}
Let $\K$ be a number field of degree $n_{\K} \geq 2$ and discriminant $\Delta_{\K}$, and let $\rho$ be
the representation introduced in Section~\ref{Section:ArtinL} that produces the Artin $L$-function
$L(s,\rho)$ such that $\zeta_\K(s) = \zeta(s) L(s,\rho)$. Moreover, let $d= n_\K - 1$ and let $a_\rho$ be
the arithmetic function
\begin{equation*}
a_\rho(n)
= \begin{cases}
    \alpha_1(p)^k + \cdots + \alpha_d(p)^k & \text{if $n=p^k$},\\
    0                                      & \text{otherwise}
  \end{cases}
\end{equation*}
in which $\alpha_1(p)$, $\alpha_2(p)$, \ldots, $\alpha_d(p)$ are the local eigenvalues associated to
$\rho$ at $p$ (see Section \ref{Section:ArtinL}), and define
\begin{equation}\label{eq:Sigma}
\Sigma(x) = \sum_{n\leq x} \frac{\Lambda(n)a_\rho(n)}{n \ln{n}} ,
\end{equation}
in which $\Lambda(n)$ is the von Mangoldt function. Throughout, any sum involving $\Lambda(n)/\ln n$ is
understood to omit the term $n = 1$. The goal in this section is to derive the upper and lower bounds for
$\Sigma(x)$ that we require in the proof of Theorem \ref{Theorem:Main}. These bounds are presented in
Lemmas \ref{lem:SigmaUpper} and \ref{lem:SigmaLower}, which we prove in Section \ref{Section:KI}.

The remainder of this section is organised as follows. In Section \ref{Section:Psi}, we establish an
explicit, conditional approximation for
\begin{equation*}
\Psi(x) = \sum_{n \leq x} \frac{\Lambda(n)}{n \ln{n}} .
\end{equation*}
In Section \ref{Section:PrimeFunctions}, we prove some auxiliary lemmas which bound certain functions
over primes. In Section \ref{Section:ArtinL}, we summarise important definitions and properties of the
Artin $L$-function $L(s,\rho)$. In Section \ref{Section:KI}, we combine the results from Sections
\ref{Section:Psi}--\ref{Section:ArtinL} to establish the explicit upper and lower bounds for $\Sigma(x)$
we require.

\subsection{Explicit bounds for \texorpdfstring{$\Psi(x)$}{Psi(x)}}\label{Section:Psi}
The following result may be of independent interest. The method used to prove it is motivated by the
related work of Ramar\'{e} \cite{RamareEESF}, whose method is similar to the approaches in Rosser and
Schoenfeld's paper \cite{Rosser}.
\begin{theorem}\label{Theorem:PsiExplicit}
If the Riemann Hypothesis (RH) is true and $x\geq e$, then
\begin{equation}\label{eq:PsiExplicitFormula}
\big|\Psi(x) - \ln\ln x  - \gamma \big|
\leq
\dfrac{3(\ln x  + 2)}{8\pi \sqrt{x}} .
\end{equation}
\end{theorem}
\begin{proof}
For $e \leq x <73.2$, a computation verifies \eqref{eq:PsiExplicitFormula}. For $x \geq 73.2$, we use a
result of Schoenfeld \cite{Schoenfeld}: if $x\geq 73.2$ and RH is true, then $\psi(x)= \sum_{n \leq x}
\Lambda(n)$ satisfies\label{p:psi}
\begin{equation}\label{eqn:Schoenfeld}
|\psi(x) - x| \leq \frac{\sqrt{x}(\ln x )^2}{8\pi} .
\end{equation}
Let $f(t)=\tfrac{1}{t\ln{t}}$. A partial summation yields:
\begin{align*}
\Psi(x)
&=\sum_{n\leq x} \Lambda(n) f(n)
 = \psi(x) f(x) - \int_2^x \psi(t) f'(t)\dd t                                            \\
&= (\psi(x) - x) f(x) + 2 f(2) + \int_2^x f(t)\dd t - \int_2^x (\psi(t) - t) f'(t)\dd t .
\end{align*}
Since the integral $\int_2^\infty (\psi(t) - t) f'(t)\dd t$ converges, this shows that
\begin{equation}\label{eqn:bho}
\Psi(x)
= \ln\ln x
  + \opn{C}
  + \frac{\psi(x) - x}{x \ln x }
  - \int_x^\infty (\psi(t)-t) \Big(\frac{1 + \ln{t}}{t^2 \ln^2{t}}\Big)\dd t ,
\end{equation}
for a suitable constant $\opn{C}$. To identify the value of $\opn{C}$, note that both terms
involving $\psi$ go to zero, so that
$
\opn{C}
= \lim_{x\to\infty} \big( \Psi(x) - \ln\ln x \big) .
$
Mertens' theorem also provides
\begin{equation*}
\prod_{p\leq x} \Big(1 - \frac{1}{p}\Big)^{-1} \sim e^{\gamma} \ln x .
\end{equation*}
Therefore, \cite[Lem.~2.3]{Lamzouri} implies
\begin{align*}
\opn{C}
&= \lim_{x\to\infty} \Big\{\Psi(x) - \ln\ln x \Big\}
 = \lim_{x\to\infty} \Big\{\ln\prod_{p\leq x} \Big(1 - \frac{1}{p}\Big)^{-1} + O\Big(\frac{1}{\sqrt{x}\ln x }\Big) - \ln\ln x \Big\} \\
&= \lim_{x\to\infty} \Big\{\ln(e^\gamma \ln x )) - \ln\ln x \Big\}
 = \gamma .
\end{align*}
For $x\geq 73.2$, use \eqref{eqn:Schoenfeld} in \eqref{eqn:bho} and obtain the desired result:
\begin{equation*}
\big|\Psi(x) - \ln\ln x - \gamma \big|
\leq \frac{1}{8 \pi} \Big[ \frac{\ln x }{\sqrt{x}} + \int_x^\infty \Big(\frac{1 + \ln{t}}{t^{3/2}}\Big)\dd t \Big]
=    \frac{3(\ln x + 2)}{8 \pi \sqrt{x}} .
\qedhere
\end{equation*}
\end{proof}

\subsection{Explicit bounds for functions over primes}\label{Section:PrimeFunctions}
The following lemmas enable us to bound certain functions over primes which arise naturally in this
paper. Lemma \ref{lem:annoyingly_long} is unconditional.
\begin{lemma}[\text{Lee--Nosal \cite[Cor.~1.3]{LeeNosal}}]
\label{lem:RS_M3}
If $x \geq 23.8$ and the RH is true, then
\begin{equation}
 \prod_{p \leq x} \Big(1-\frac{1}{p}\Big) \geq \frac{e^{-\gamma}}{\ln{x}} \Big(1 - \frac{3\ln{x}}{8\pi\sqrt{x}}\Big).
\end{equation}
\end{lemma}
The previous lemma improves upon \cite[(6.23)]{Schoenfeld}, in which $3 \ln x$ is replaced with $3 \ln x
+ 5$, although the range of applicability for that result is $x \geq 8$.
\begin{lemma}\label{lem:annoyingly_long}
If $x\geq 59$ and $n \geq 2$, then
\begin{equation*}
\prod_{p \leq x} \Big(1-\frac{1}{p^n}\Big)^{-1}
\geq \zeta(n) \exp\Big(- \frac{1}{x \ln x}\Big(1 + \frac{1/2}{\ln x}\Big) \Big) .
\end{equation*}
\end{lemma}
\begin{proof}
First, we prove that if $x\geq 59$, then
\begin{equation}\label{eq:A4}
-\sum_{p > x} \ln\left(1-\frac{1}{p^2}\right)
\leq \frac{1}{x\ln x}\Big(1 + \frac{1/2}{\ln x} \Big).
\end{equation}
Let $\pi(x)$ be the prime-counting function. For $x\geq 59$, \cite[Thm.~1]{Rosser} yields
\begin{equation*}
\pi(x) = \frac{x}{\ln x} + R(x),
\end{equation*}
in which $R^-(x) \leq R(x) \leq R^+(x)$ with
\begin{equation}\label{eqn:ChebyshevBounds}
R^-(x) = \frac{x}{2 \ln^2 x}
\qquad \text{and} \qquad
R^+(x) = \frac{3x}{2 \ln^2 x};
\end{equation}
the upper bound holds for $x>1$.
By partial summation we get
\begin{align*}
-\sum_{p > x} \ln\Big(1-\frac{1}{p^2}\Big)
&=  \pi(x)\ln\Big(1-\frac{1}{x^2}\Big) + 2 \int_{x}^{\infty} \frac{\pi(t)}{t^3-t}\dd t	\\
&\leq - \frac{\pi(x)}{x^2} + 2 \int_{x}^{\infty} \frac{\pi(t)}{t^3}\dd t + 2\int_{x}^{\infty} \frac{\pi(t)\dd t}{t^3(t^2-1)} .
\intertext{For convenience, let $\opn{C}=\int_{x}^{\infty} \frac{2\pi(t)\dd t}{t^3(t^2-1)}$ denote the
last term in the previous formula. Continuing, this formula and \eqref{eqn:ChebyshevBounds} produce the upper
bound}
&= - \frac{1}{x \ln x} + 2 \int_x^{\infty} \frac{\dd t}{t^2 \ln t} - \frac{R(x)}{x^2} + 2 \int_x^{\infty}\frac{R(t)}{t^3}\dd t + \opn{C}	\\
&= - \frac{1}{x \ln x} {+} 2\Big(\frac{1}{x\ln x} {-} \int_x^{\infty} \frac{\dd t}{t^2\ln^2 t}\Big)
{-} \frac{R(x)}{x^2} {+} 2\int_x^{\infty} \frac{R(t)}{t^3}\dd t {+} \opn{C}	\\
&\leq \frac{1}{x \ln x} - 2\int_x^{\infty} \frac{\dd t}{t^2 \ln^2 t} - \frac{R^-(x)}{x^2} + 2\int_x^{\infty} \frac{R^+(t)}{t^3}\dd t + \opn{C}	\\
&= \frac{1}{x \ln x} - \frac{1/2}{x \ln^2x} + \int_x^{\infty} \frac{\dd t}{t^2 \ln^2 t} + \opn{C}	\\
&= \frac{1}{x \ln x} + \frac{1/2}{x \ln^2x} - \int_x^{\infty} \frac{2\dd t}{t^2 \ln^3 t} + \int_{x}^{\infty} \frac{2\pi(t)\dd t}{t^3(t^2-1)}	\\
&\leq \frac{1}{x \ln x} + \frac{1/2}{x \ln^2x},
\end{align*}
since $\frac{\pi(t)}{t^3(t^2-1)}\leq \frac{1}{t^2\ln^3 t}$.
The desired result follows from \eqref{eq:A4} and the fact that
\begin{equation*}
\prod_{p \leq x} \Big(1-\frac{1}{p^{n}}\Big)^{-1}
= \zeta(n) \prod_{p > x} \Big(1-\frac{1}{p^n}\Big)
\geq \zeta(n) \prod_{p > x} \Big(1-\frac{1}{p^2}\Big).
\qedhere
\end{equation*}
\end{proof}

\subsection{Artin \texorpdfstring{$L$}{L}-functions}\label{Section:ArtinL}
We require a handful of well-known properties of the following specific Artin $L$-function; see
\cite{Serre, Neukirch, WongThesis} for more details.

Let $\L$ be the normal closure of $\K$, $G=\Gal(\L/\Q)$, and $H=\Gal(\L/\K)$. Let $\triv_G$ and $\triv_H$
be the trivial representations of $G$ and $H$, respectively. We consider $\Ind_H^G \triv_H$, the induced
representation of $\triv_H$ from $H$ to $G$. It is well known that $\zeta_\K(s) = L(s, \triv_H; \L/\K) =
L(s, \Ind_H^G \triv_H; \L/\Q)$ and that $\zeta(s) = L(s, \triv_G; \L/\Q)$. Since $\triv_G$ is obviously
contained in $\Ind_H^G \triv_H$, there exists a representation $\rho$ of $G$ such that $\Ind_H^G
\triv_H=\triv_G\oplus \rho$. The multiplicativity of Artin's $L$-functions with respect to the direct sum
of representations shows that
\begin{equation*}
\zeta_\K(s)
 = L(s, \Ind_H^G \triv_H; \L/\Q)
 = L(s, \triv_G\oplus \rho; \L/\Q)
 = \zeta(s) L(s,\rho; \L/\Q).
\end{equation*}
We denote $L(s,\rho)=L(s,\rho;\L/\Q)$ from now on.
The previous equality shows that $L(1,\rho)= \kappa_\K$ and that $L(s,\rho)$ is entire if and only if the
quotient $\zeta_\K/\zeta$ is entire.
Finally, comparing the functional equations of $\zeta_\K$, $\zeta$ and $L(s,\rho)$, one recovers that the
dimension and Artin conductor of $\rho$ are $n_{\K}-1$ and $|\Delta_{\K}|$, respectively.
Note that the dimension is obvious from the induction.

For $\Re s > 1$ we have also the expansion
\begin{equation*}
L(s, \rho) = \prod_p \prod_{i=1}^d \Big(1 - \frac{\alpha_i(p)}{p^s} \Big)^{-1} ,
\end{equation*}
in which $d = \deg{\rho} = n_{\K} - 1$, the product runs over primes $p$, and $\alpha_i(p)$ are the
local eigenvalues associated to $\rho$ at $p$. Taking the logarithm of $L(s,\rho)$, we have
\begin{equation*}
\ln L(s, \rho)
= \sum_{n=2}^\infty \frac{\Lambda(n) a_{\rho}(n)}{n^s \ln n} ,
\end{equation*}
in which $a_{\rho}(p^k) = \alpha_1(p)^k + \cdots + \alpha_d(p)^k$. Every $\alpha_i(p)$ is a root of unity
or zero, so $|\alpha_i(p)| \leq 1$ and hence $|a_\rho(p^k)| \leq d = n_{\K} - 1$, for every prime power
$p^k$.

\subsection{Key identities}\label{Section:KI}
In this subsection, we derive upper and lower bounds for $\Sigma(x)$, which was defined in
\eqref{eq:Sigma}. Our upper bound is presented in the next lemma.
\begin{lemma}\label{lem:SigmaUpper}
Suppose the RH is true. If $x\geq e$, then
\begin{equation}
\Sigma(x) \leq (n_{\K} - 1) \Big(\ln\ln{x} + \gamma + \dfrac{3(\ln x + 2)}{8\pi \sqrt{x}}\Big) .
\end{equation}
\end{lemma}
\begin{proof}
The bound $|a_\rho(n)| \leq n_{\K}-1$ for prime powers implies $\Sigma(x) \leq (n_{\K} - 1)\Psi(x)$.
Applying Theorem \ref{Theorem:PsiExplicit} to this upper bound gives the result.
\end{proof}
Our argument to obtain a lower bound on $\Sigma(x)$ is more involved and includes an intermediate
computation that can be made more precise by restricting the range of $x$. In the end, we only require
bounds for $\Sigma(x)$ on $x \geq 5\cdot 10^5$.
\begin{lemma}\label{lem:SigmaLower}
Suppose the RH is true. If $x \geq 59$, then
\begin{equation*}
\Sigma(x)
\geq \ln\bigg(\frac{\zeta(n_{\K})}{e^{\gamma}\ln{x}} \Big(1 - \frac{3\ln{x}}{8\pi\sqrt{x}}\Big)\bigg)
    - \frac{n_{\K} - 1}{\ln{x}} \Big(1 + \frac{\ln^2 x}{8\pi\sqrt{x}} + \frac{1.3\ln{x}}{x}\Big) .
\end{equation*}
\end{lemma}
\begin{proof}
Follow the steps in \cite[pp.~373-374]{ChoKim} to see that
\begin{equation*}
\Sigma(x)
= \sum_{n\leq x} \frac {\Lambda(n)a_\rho(n)}{n\ln n}
= -\sum_{p\leq x} \sum_{i=1}^d \ln\!\Big(1-\frac{\alpha_i(p)}{p}\Big) + d\sum_{p\leq x} A_p ,
\end{equation*}
where $d = n_{\K}-1$ and
\begin{equation*}
A_p = -\frac{1}{d} \sum_{\substack{k\geq 1\\ p^k > x}} \frac{a_\rho(p^k)}{k p^k}.
\end{equation*}
Then,
\begin{equation*}
\sum_{p\leq x} |A_p|
\leq \frac{1}{x\ln x} \sum_{p\leq x} \frac {\ln p}{1-p^{-1}}
= \frac{1}{x\ln x} \Big(\sum_{p\leq x} \ln{p}
 + \sum_{p\leq x} \frac{\ln p}{p}
 + \sum_{p\leq x} \frac{\ln p}{p(p-1)}\Big) .
\end{equation*}
If $x\geq 59$ and the RH is true, then \cite[(6.5)]{Schoenfeld}, \cite[(3.23)]{Rosser}, and the estimate
$\sum_{p} \frac{\ln p}{p(p-1)} \leq 0.8$ tell us
%
\begin{equation*}
\sum_{p\leq x} \frac{\ln{p}}{1-p^{-1}}
\leq x + \frac{\sqrt{x}\ln^2x}{8\pi}
 + \ln{x}.
\end{equation*}
%
Therefore,
\begin{align*}
\sum_{p\leq x} |A_p|
\leq \frac{1}{\ln{x}}\Big(1 + \frac{\ln^2 x}{8\pi\sqrt{x}} + \frac{\ln{x}}{x} \Big) .
\end{align*}
Next, for all primes $p$, \cite[(3.4)]{ChoKim} tells us that
\begin{equation}\label{euler}
\prod_{i=1}^d \Big(1-\frac{\alpha_i(p)}{p}\Big)^{-1}
\geq \Big(1-\frac{1}{p}\Big)\Big(1-\frac{1}{p^{d+1}}\Big)^{-1} .
\end{equation}
Hence
\begin{equation}\label{eq:SigBoE28}
\Sigma(x)
\geq \ln\bigg(\prod_{p\leq x} \Big(1 - \frac{1}{p}\Big) \prod_{p\leq x} \Big(1 - \frac{1}{p^{d+1}}\Big)^{-1} \bigg)
      - \frac{d}{\ln{x}} \Big(1 + \frac{\ln^2 x}{8\pi\sqrt{x}} + \frac{\ln{x}}{x} \Big) .
\end{equation}
Since $d + 1 = n_{\K}$ and $x\geq 59$, Lemma \ref{lem:annoyingly_long} ensures that
\begin{equation*}
\prod_{p\leq x} \Big(1-\frac{1}{p^{n_{\K}}}\Big)^{-1}
\geq \zeta(n_\K)\exp\Big(-\frac{0.3}{x}\Big) .
\end{equation*}
%
To get the desired result, apply this lower bound and Lemma \ref{lem:RS_M3} in \eqref{eq:SigBoE28}.
\end{proof}

\section{Explicit short sum relationship}\label{Section:Duck}
The following result approximates $\ln{\kappa_{\K}}$ with $\Sigma(x)$, for which bounds were established
in the previous section.
\begin{theorem}\label{Theorem:ShortSum}
If $\K\neq\Q$, GRH for $\zeta_\K$ is true, and $\zeta_{\K}/\zeta$ is entire, then for $x \geq 5\cdot
10^5$, we have
\begin{align*}
\big|\ln{\kappa_{\K}} - \Sigma(x)\big|
&\leq \Big(\frac{3e}{8\pi} + \frac{1.45}{\ln{x}}\Big) \frac{(\ln{x})^3}{\sqrt{x}} (n_{\K} - 1)
    + \Big(\frac{3e}{4\pi} + \frac{6.01}{\ln{x}}\Big) \frac{(\ln{x})^2}{\sqrt{x}} \ln|\Delta_{\K}| .
\end{align*}
\end{theorem}
The remainder of this section is devoted to the proof of Theorem \ref{Theorem:ShortSum}, which spans
Sections \ref{ssec:StepI}--\ref{ssec:StepV}. To begin, we assume the following:
\begin{equation}\label{eq:Parameters}
x \geq 5\cdot 10^5,
\quad \sqrt{x}(\ln x)^2 \leq h < x,
\quad c = \frac{1}{\ln x},
\quad d = n_{\K}-1,
\quad \alpha = \frac{1}{2} + \epsilon,
\end{equation}
in which $\epsilon \in (0,1/2-c)$ can be made small independently of the other parameters.
The lower bound $x \geq 5\cdot 10^5$ is used only in the final stage of the proof.

\subsection{Step I}\label{ssec:StepI}
Integrate \eqref{eq:Sigma}, use partial summation, and obtain
\begin{equation*}
\int_0^{x} \Sigma(t)\dd t
= \sum_{n \leq x} \frac{\Lambda(n) a_\rho(n) (x - n)}{n \ln{n}} .
\end{equation*}
Thus,
\begin{align*}
\frac{1}{h} \int_x^{x+h} \Sigma(t)\dd t
&= \frac{1}{h} \Big(\sum_{n \leq x + h} \frac{\Lambda(n) a_\rho(n) (x + h - n)}{n \ln{n}} - \sum_{n \leq x} \frac{\Lambda(n) a_\rho(n) (x - n)}{n \ln{n}}\Big) \\
&= \Sigma(x) + \frac{1}{h} \sum_{x < n \leq x+h} \frac{\Lambda(n) a_\rho(n) (x+h - n)}{n \ln{n}} .
\end{align*}
Since $|a_\rho(n)| \leq d$ for prime powers and $0\leq x+h - n < h$ for $x < n \leq x+h$,
\begin{align}
\bigg|\Sigma(x) - \frac{1}{h} \int_x^{x+h} \Sigma(t)\dd t\bigg|
&\leq \frac{1}{h} \bigg|\sum_{x < n \leq x+h} \frac{\Lambda(n) a_\rho(n) (x+h - n)}{n \ln{n}} \bigg| \nonumber\\
&\leq d\sum_{x < n \leq x+h} \frac{\Lambda(n) }{n \ln{n}}
= d ( \Psi(x+h) - \Psi(x) )                                                                          \nonumber\\
&\leq d\Big(\int_x^{x+h} \frac{\dd y}{y\ln{y}} + \frac{3(\ln(x+h) + 2)}{4\pi \sqrt{x}} \Big)         \nonumber\\
&\leq d\Big(\frac{h}{x\ln x} + \frac{3(\ln x + \tfrac{h}{x} + 2)}{4\pi \sqrt{x}} \Big)               \label{eq:Hellscape}
\end{align}
by Theorem \ref{Theorem:PsiExplicit} and the inequality $\ln(x+h)= \ln x + \ln (1+\frac{h}{x} )
\leq \ln x + \frac{h}{x}$.

\subsection{Step II}\label{ssec:StepII}
Observe that
\begin{equation}\label{eq:LogMangle}
\ln{L(1+s,\rho)} = \sum_{n=2}^{\infty} \frac{\Lambda(n) a_\rho(n)}{n^{1+s} \ln n}
\end{equation}
converges absolutely on $\Re{s} > 0$. Since $x > 0$ and $c > 0$, \cite[Thm.~5.1]{MontgomeryVaughan} implies
\begin{align*}
\sideset{}{^{\prime}}\sum_{n \leq x} \frac{\Lambda(n) a_\rho(n)}{n \ln{n}}
&= \frac{1}{2\pi i} \int_{c - i\infty}^{c + i\infty} \ln{L(1+s,\rho)} \frac{x^s}{s}\dd s ,
\end{align*}
where the prime indicates that if $x$ is an integer, then the last term in the sum is counted with weight
$1/2$. Therefore,
\begin{equation*}
\Sigma(x) = \sum_{n\leq x} \frac{\Lambda(n) a_\rho(n)}{n \ln{n}}
= \frac{1}{2\pi i} \int_{c - i\infty}^{c + i\infty} \ln{L(1+s,\rho)} \frac{x^s}{s}\dd s
\end{equation*}
except, at most, at integral $x$. This does not affect the mean value below, so
\begin{equation}\label{eq:SmoothOperator}
\frac{1}{h} \int_{x}^{x+h} \!\!\!\! \Sigma(t)\dd t
= \frac{h^{-1}}{2\pi i} \int_{c - i\infty}^{c + i\infty} \ln{L(1+s,\rho)} \frac{(x+h)^{s+1} - x^{s+1}}{s(s+1)}\dd s .
\end{equation}
Finally, \eqref{eq:Hellscape} and \eqref{eq:SmoothOperator} imply that
\begin{equation}\label{eq:Helmet}
\Big|\Sigma(x) - \frac{h^{-1}}{2\pi i} \int_{c - i\infty}^{c + i\infty} \ln{L(1+s,\rho)} \frac{(x+h)^{s+1} - x^{s+1}}{s(s+1)}\dd s \Big|
\leq d \opn{G}(x,h) ,
\end{equation}
in which
\begin{equation}\label{eq:Hell}
\opn{G}(x,h) = \frac{h}{x\ln x} + \frac{3(\ln x+ \tfrac{h}{x} + 2)}{4\pi \sqrt{x}}.
\end{equation}

\subsection{Step III}\label{ssec:StepIII}
The following lemma is an explicit version of \cite[Lem.~8.1]{GranvilleSound}, but it is formulated
explicitly for the representation $\rho$ we are considering in this section.

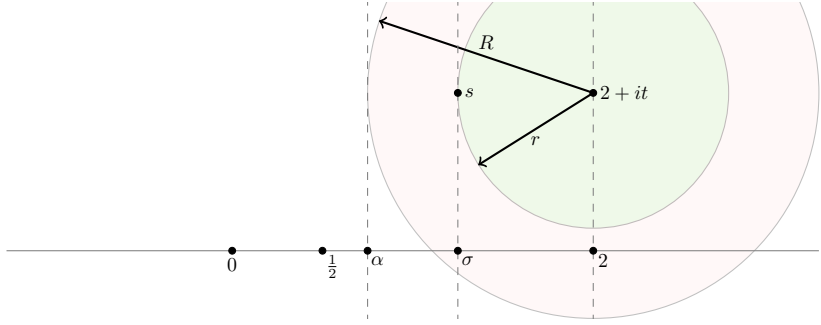
\begin{figure}[h]
    \begin{tikzpicture}[scale=0.6, every node/.style={scale=0.75}]
        \clip(-5,-1.5) rectangle (13.5,5.5);

        \filldraw[fill=red!10!white, opacity=0.25] (8,3.5) circle (5cm);
        \filldraw[fill=green!25!white,opacity=0.25] (8,3.5) circle (3cm);
        \draw[thick,->] (8,3.5)--(5.45,1.9) node[midway, below]{$r$};
        \draw[thick,->] (8,3.5)--(3.25,5.1) node[midway, above]{$R$};

       \draw[thin,gray](-5,0)--(13,0);
       \draw[thin,gray,dashed](3,-3)--(3,6);
       \draw[thin,gray,dashed](5,-3)--(5,6);
       \draw[thin,gray,dashed](8,-3)--(8,6);

        \filldraw[fill=black] (0,0) circle (0.075cm)node[below]{$0$};
        \filldraw[fill=black] (2,0) circle (0.075cm)node[xshift=5, yshift=-8]{$\frac{1}{2}$};
        \filldraw[fill=black] (3,0) circle (0.075cm)node[xshift=5, yshift=-5]{$\alpha$};
        \filldraw[fill=black] (5,0) circle (0.075cm)node[xshift=5, yshift=-5]{$\sigma$};
        \filldraw[fill=black] (5,3.5) circle (0.075cm)node[right]{$s$};
        \filldraw[fill=black] (8,0) circle (0.075cm)node[xshift=5, yshift=-5]{$2$};
        \filldraw[fill=black] (8,3.5) circle (0.075cm)node[right]{$2+it$};
    \end{tikzpicture}
\caption{Proof of Lemma \ref{Lemma:Handy}.}
\label{Figure:Circles}
\end{figure}

\begin{lemma}\label{Lemma:Handy}
If GRH is true, $L(s,\rho)$ is entire, and $s=\sigma + it$ with $\frac{1}{2} < \alpha < \sigma < 2$,
then
\begin{equation*}
|\ln{L(s,\rho)}| \leq \frac{\opn{H}(\alpha,t) d + (2-\alpha) \ln|\Delta_{\K}|}{\sigma - \alpha} ,
\end{equation*}
in which
\begin{equation}\label{eq:C2}
\opn{H}(\alpha,t) = (4-2\alpha) \bigg[\ln \bigg( \frac{\pi ^{3/2} \zeta (\frac{3}{2})}{6 \sqrt{2}} \bigg) + \frac{\ln(2-\alpha + |3+it|)}{2}\bigg] .
\end{equation}
\end{lemma}
\begin{proof}
Consider the circles with centre $2+it$ and radii $r=2-\sigma < R = 2-\alpha$, so that the smaller circle
passes through $s$; see Figure \ref{Figure:Circles}.
%
%
By GRH and the entireness of $\ln{L(s,\rho)}$, the closed disk $|z - (2 + it)| \leq R$ is zero-free.
Hence the branch of $\ln L(z,\rho)$ determined by its Dirichlet-series branch on $\Re z > 1$ extends
analytically to a neighborhood of this disk.
If $\Re{s} \geq \frac{1}{2}$, then \cite[Lem.~5]{GarciaLeeArtinL} tells us
\begin{equation*}
|L(s,\rho)|
\leq \bigg(\frac{\zeta(\frac{3}{2})}{\sqrt{2 \pi}} \bigg)^d \sqrt{|\Delta_{\K}|}|1+s|^{\frac{d}{2}}.
\end{equation*}
Then $\ln|L(z,\rho)| = \Re \ln {L(z,\rho)}$ and \cite[eq.~(8)]{GarciaLeeArtinL} ensure that
\begin{equation*}
|\ln{L(z,\rho)}| \leq d \ln \zeta(\delta)
\end{equation*}
for $\Re{z} \geq \delta > 1$. The Borel--Carath\'{e}odory theorem says that any point on the smaller circle
(in particular $s$) satisfies
\begin{align*}
|\ln L(&s,\rho)|
 \leq \frac{2r}{R-r} \max_{|z-(2+it)|=R} \Re{\ln{L(z,\rho)}} + \frac{R+r}{R-r} |\ln{L(2+it,\rho)}| \\
&\leq \frac{4 - 2\sigma}{\sigma -\alpha} \max_{|z-(2+it)|=R} \Re{\ln{L(z,\rho)}} + \frac{(4-2\alpha) \ln{\zeta(2)}}{\sigma -\alpha} d \\
&\leq \frac{4-2\alpha}{\sigma -\alpha} \bigg[d\Big(\ln \bigg(\frac{\zeta(\frac{3}{2})}{\sqrt{2 \pi}} \bigg)  + \ln \frac{\pi^2}{6} + \frac{1}{2} \max_{|z-(2+it)|=R} \ln{|1+z|}\Big) + \frac{\ln|\Delta_{\K}|}{2}\bigg],
\end{align*}
in which $|1+z| \leq |3+it| + R = 2 -\alpha +|3+it|$.
\end{proof}
\begin{remark}
Above we applied \cite[Lem.~5]{GarciaLeeArtinL}. One should be able to prove an upper bound for
$|L(s,\rho)|$ involving the factor $|1+s|^{d (1-\sigma)/2}$ using the Phragm\'{e}n--Lindel\"{o}f principle
\cite{PalZhao}. We do not pursue this since the existing bound suffices for our purposes and our final
steps provide tighter numerical control.
\end{remark}

\subsection{Step IV}\label{ssec:StepIV}
We now use Lemma \ref{Lemma:Handy} to move the line of integration of the integral in \eqref{eq:Helmet}.
We remind the reader of our choice of parameters \eqref{eq:Parameters}.
\begin{lemma}\label{Lemma:TheIntegral}
Suppose the GRH is true, $L(s,\rho)$ is entire, and $h = x^{1 - \mu}$ for some real parameter
$0<\mu\leq \frac{1}{2} - \frac{2\ln\ln x}{\ln x}$.
Then, we have
\begin{align}
\Big|\frac{h^{-1}}{2\pi i}&\int_{c-i\infty}^{c+i\infty} \ln{L(1+s,\rho)} \frac{(x+h)^{s+1} - x^{s+1}}{s(s+1)}\dd s - \ln{L(1,\rho)}\Big| \nonumber\\
&\leq \opn{J}(x,\mu) \frac{e  (\ln{x})^2}{\pi \sqrt{x}} \Big(\opn{H}(\tfrac{1}{2},x^{\mu})  \Big(1 + \frac{1}{\mu\ln{x}}\Big)d + \frac{3}{2} \ln|\Delta_{\K}| \Big) , \label{eq:IntegralDamn}
\end{align}
in which
\begin{equation}\label{eq:Damn}
\begin{split}
\opn{J}(x,\mu)
&= \Big(\mu + \frac{2}{\ln{x} - 2}\Big) + \frac{2 \big(1 + x^{-\mu} \big)^{\frac{1}{2} + \frac{1}{\ln{x}}}}{\ln{x}} .
\end{split}
\end{equation}
\end{lemma}

\begin{figure}
    \centering
    \begin{tikzpicture}[xscale=1.25, yscale=1.0]

      \coordinate (A) at (6,-1);           
      \coordinate (B) at (6,1);            
      \coordinate (C) at (0,1);        
      \coordinate (D) at (0,-1);       

      \begin{scope}[very thick,decoration={
    markings,
    mark=at position 0.5 with {\arrow{>}}}
    ]
      \draw[postaction={decorate}] (A) -- (B) node[midway,right=3pt] {$\Gamma_1$};
      \draw[postaction={decorate}] (B) -- (C) node[midway,above=3pt] {$\Gamma_2$};
      \draw[postaction={decorate}] (C) -- (D) node[midway,left=3pt] {$\Gamma_3$};
      \draw[postaction={decorate}] (D) -- (A) node[midway,below=3pt] {$\Gamma_4$};
    \end{scope}
      \node[below] at (A) {$c - iT$};
      \node[above] at (B) {$c + iT$};
      \node[above] at (C) {$\alpha - 1+ c + iT$};
      \node[below] at (D) {$\alpha - 1+c - iT$};

    \end{tikzpicture}
    \caption{The contours $\Gamma_1$, $\Gamma_2$, $\Gamma_3$, $\Gamma_4$.}
    \label{Figure:Contour}
\end{figure}
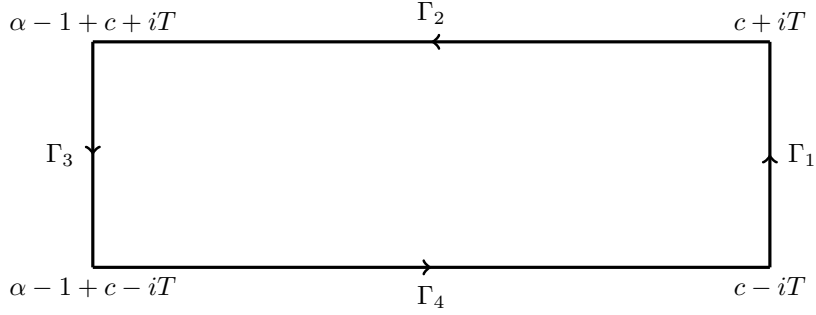

\begin{proof}
Since $\alpha = \frac{1}{2} + \epsilon$, in which $0 < \epsilon < \frac{1}{2}$,\label{p:Alpha}
the GRH ensures that $L(s,\rho)$ is zero-free in $[\alpha, 1] \times (-\infty,\infty)$.
Let $\mathcal{R} = \Gamma_1\cup\Gamma_2\cup\Gamma_3\cup\Gamma_4$ as in Figure \ref{Figure:Contour}. Define
\begin{align}
I_{\Gamma_i}
&= \frac{h^{-1}}{2\pi i }\int_{\Gamma_i} \ln{L(1+s,\rho)} \frac{(x+h)^{s+1} - x^{s+1}}{s(s+1)}\dd s \nonumber\\
&= \frac{h^{-1}}{2\pi i }\int_{\Gamma_i} \frac{\ln{L(1+s,\rho)}}{s} \int_x^{x+h} y^{s}\dd y\dd s    \label{eq:IntLine}
\end{align}
and note that $I_{\Gamma_1}$ is
\begin{equation}\label{eqn:bristol2}
\frac{h^{-1}}{2\pi i }\int_{c-iT}^{c+iT} \ln{L(1+s,\rho)} \frac{(x+h)^{s+1} - x^{s+1}}{s(s+1)}\dd s
= I_{\mathcal{R}} - I_{\Gamma_2} - I_{\Gamma_3} - I_{\Gamma_4}.
\end{equation}
The zero-free region $[\alpha, 1] \times [-T,T]$ ensures that the only singularity inside the closed
contour $\mathcal{R}$ is at $s=0$ (because the hypotheses imply that $\alpha-1+c<0$), so the residue
theorem implies $I_{\mathcal{R}} = \ln{L(1,\rho)}$.

\medskip\noindent\textsc{Bounding $I_{\Gamma_2}$ and $I_{\Gamma_4}$.}
For $s = \sigma \pm iT$ on $\Gamma_2$ and $\Gamma_4$,
\begin{equation*}
\frac{1}{\ln x} = c = 1 + (\alpha - 1 + c) - \alpha \leq (1+ \sigma) - \alpha,
\end{equation*}
so \eqref{eq:IntLine}, Lemma \ref{Lemma:Handy}, and symmetry ensure that
\begin{align*}
|I_{\Gamma_2}|+|I_{\Gamma_4}|
&\leq \frac{1}{\pi h } \Big| \int_{\alpha - 1 + c + iT}^{c + iT}  \frac{\ln{L(1+s,\rho)}}{s} \int_{x}^{x+h} t^{s} \dd t \dd s \Big|  \\
&\leq \frac{(x+h)^c}{\pi T} \int_{\alpha - 1 + c}^{c}
| \ln{L(1+\sigma + iT,\rho)} | \dd\sigma \\
&\leq \frac{(x+h)^c}{\pi T} \int_{\alpha - 1 + c}^{c}
\frac{\opn{H}(\alpha,T) d + (2-\alpha) \ln|\Delta_{\K}|}{1 + \sigma - \alpha} \dd\sigma \\
&\leq \frac{(x+h)^c (\opn{H}(\alpha,T) d + (2-\alpha) \ln|\Delta_{\K}|)}{\pi T c} \int_{\alpha - 1 + c}^{c}
\dd\sigma
\to 0
\end{align*}
as $T \to \infty$, {since $\opn{H}(\alpha,T) \ll \ln{T}$}.  Thus, $I_{\Gamma_2}$ and $I_{\Gamma_4}$
vanish in the limit.

\medskip\noindent\textsc{Bounding $I_{\Gamma_3}$.}
Suppose that $0 < \mu < 1$ is a real parameter to be chosen later such that $x^{\mu} > 1$.
As $T\to\infty$, \eqref{eq:IntLine} ensures that $|I_{\Gamma_3}| \leq |J_1| + |J_2|$, where
\begin{align*}
J_1 &= \frac{h^{-1}}{2\pi i} \int_{- x^{\mu}}^{x^{\mu}} \ln{L(\alpha + c + it,\rho)} \int_x^{x+h} \frac{y^{\alpha - 1 + c + it}}{\alpha - 1 + c + it}\dd y\dd t , \quad\text{and} \\
J_2 &= \frac{h^{-1}}{2\pi i} \Bigg(\int_{-\infty}^{-x^{\mu}} + \int_{x^{\mu}}^{\infty}\Bigg)
                             \ln{L(\alpha + c + it,\rho)} \frac{(x+h)^{\alpha + c + it} - x^{\alpha + c + it}}{(\alpha - 1 + c + it)(\alpha + c + it)}\dd t .
\end{align*}
We begin by bounding $J_1$. First, observe that
\begin{equation*}
\int_x^{x+h} y^{\alpha - 1 + c}\dd y
\leq x^{\alpha - 1 + c} \int_x^{x+h} \dd y
=    h x^{\alpha - 1 + c} .
\end{equation*}
It follows from this and Lemma \ref{Lemma:Handy} that
\begin{align*}
|J_1|
&\leq \frac{1}{\pi h c} \int_{0}^{x^{\mu}} \frac{\opn{H}(\alpha,t) d + (2-\alpha) \ln|\Delta_{\K}|}{|\alpha - 1 + c + it|} \int_x^{x+h} y^{\alpha - 1 + c}\dd y\dd t \\
&\leq \frac{x^{\alpha - 1 + c}}{\pi c} \int_{0}^{x^{\mu}} \frac{\opn{H}(\alpha,t) d + (2-\alpha) \ln|\Delta_{\K}|}{|\alpha - 1 + c + it|} \dd t .
\end{align*}
We also have $\opn{H}(\alpha,t) \leq \opn{H}(\alpha,x^{\mu})$ on $|t| \leq x^{\mu}$ and
\begin{align*}
\int_{0}^{x^{\mu}} \frac{\dd t}{|\alpha - 1 + c + it|}
&\leq \int_{1}^{x^{\mu}} \frac{\dd t}{t} + \int_{0}^{1} \frac{\dd t}{1 - \alpha - c}
= \mu \ln{x} + \frac{1}{1 - \alpha - c} .
\end{align*}
Therefore,
\begin{align*}
|J_1|
&\leq \frac{x^{\alpha - 1 + c} (\opn{H}(\alpha,x^{\mu}) d + (2-\alpha) \ln|\Delta_{\K}|)}{\pi c} \int_{0}^{x^{\mu}} \frac{\dd t}{|\alpha - 1 + c + it|} \\
&\leq \frac{x^{\alpha - 1 + c} \ln{x} (\opn{H}(\alpha,x^{\mu}) d + (2-\alpha) \ln|\Delta_{\K}|)}{\pi} \Big(\mu \ln{x} + \frac{1}{1 - \alpha - c}\Big) .
\end{align*}
It also follows from Lemma \ref{Lemma:Handy} that
\begin{align*}
|J_2| &\leq \frac{2 (x+h)^{\alpha + c}}{\pi h c} \int_{x^{\mu}}^{\infty} \frac{\opn{H}(\alpha,t) d + (2-\alpha) \ln|\Delta_{\K}|}{t^2}\dd t .
\end{align*}
Since $\opn{H}(\alpha,t)/\ln{t}$ decreases on $t \geq x^{\mu} > 1$, we have
\begin{equation*}
\opn{H}(\alpha,t)
= \frac{\opn{H}(\alpha,t)}{\ln{t}} \ln{t}
\leq \frac{\opn{H}(\alpha,x^{\mu})}{\mu\ln{x}} \ln{t} .
\end{equation*}
Therefore,
\begin{align*}
|J_2|
&\leq \frac{2 (x+h)^{\alpha + c}}{\pi h c} \Big(\frac{\opn{H}(\alpha,x^{\mu}) d}{\mu\ln{x}} \int_{x^{\mu}}^{\infty} \frac{\ln{t}}{t^2}\dd t + (2-\alpha) \ln|\Delta_{\K}| \int_{x^{\mu}}^{\infty} \frac{\dd t}{t^2}\Big) \\
&= \frac{2 (x+h)^{\alpha + c} \ln{x}}{\pi h} \Big(\frac{\opn{H}(\alpha,x^{\mu}) d}{x^{\mu}} \Big(1 + \frac{1}{\mu\ln{x}}\Big) + \frac{(2-\alpha) \ln|\Delta_{\K}|}{x^{\mu}} \Big) .
\end{align*}
Combing these bounds, we have shown
\begin{align*}
|I_{\Gamma_3}|
\leq& \frac{x^{\alpha - 1 + c} \ln{x} (\opn{H}(\alpha,x^{\mu}) d + (2-\alpha) \ln|\Delta_{\K}|)}{\pi} \Big(\mu \ln{x} + \frac{1}{1 - \alpha - c}\Big) \\
& + \frac{2 x^{\alpha + c} \ln{x}}{\pi x^{\mu} h} \Big(1+\frac{h}{x}\Big)^{\alpha + c} \Big(\opn{H}(\alpha,x^{\mu}) \Big(1 + \frac{1}{\mu\ln{x}}\Big) d + (2-\alpha) \ln|\Delta_{\K}| \Big) .
\end{align*}

\medskip\noindent\textsc{Combination.}
Recall that $c = 1/\ln x$ (so $x^c = e$) and $\alpha = 1/2 + \epsilon$ for some $\epsilon \in (0,1/2-c)$;
see \eqref{eq:Parameters}. Let $T\to\infty$ in \eqref{eqn:bristol2} and deduce
\begin{align*}
\Big|\frac{h^{-1}}{2\pi i}&\int_{c-i\infty}^{c+i\infty} \ln{L(1+s,\rho)} \frac{(x+h)^{s+1} - x^{s+1}}{s(s+1)}\dd s - \ln{L(1,\rho)}\Big| \\
\leq& \frac{x^{\alpha - 1 + c} \ln{x} (\opn{H}(\alpha,x^{\mu}) d + (2-\alpha) \ln|\Delta_{\K}|)}{\pi} \Big(\mu \ln{x} + \frac{1}{1 - \alpha - c}\Big) \\
&    + \frac{2 x^{\alpha + c} \ln{x}}{\pi x^{\mu} h} \Big(1+\frac{h}{x}\Big)^{\alpha + c} \Big(\opn{H}(\alpha,x^{\mu}) \Big(1 + \frac{1}{\mu\ln{x}}\Big) d + (2-\alpha) \ln|\Delta_{\K}| \Big) .
\end{align*}
Set $h = x^{1 - \mu}$ and insert the definitions of $\alpha$ and $c$ in this upper bound. Upon letting
$\epsilon \to 0^+$, it follows that this upper bound converges to
\begin{align*}
\Big|\frac{h^{-1}}{2\pi i}&\int_{c-i\infty}^{c+i\infty} \ln{L(1+s,\rho)} \frac{(x+h)^{s+1} - x^{s+1}}{s(s+1)}\dd s - \ln{L(1,\rho)}\Big|       \\
\leq& \frac{e (\ln{x})^2 (\opn{H}(\tfrac{1}{2},x^{\mu}) d + \frac{3}{2} \ln|\Delta_{\K}|)}{\pi \sqrt{x}} \Big(\mu + \frac{2}{\ln{x} - 2}\Big) \\
&   + \frac{2 e \ln{x}}{\pi \sqrt{x}} \big(1 + x^{-\mu}\big)^{\tfrac{1}{2} + \frac{1}{\ln x}} \Big(\opn{H}(\tfrac{1}{2},x^{\mu}) \Big(1 + \frac{1}{\mu\ln{x}}\Big) d + \frac{3}{2} \ln|\Delta_{\K}| \Big) ,
\end{align*}
which is marginally stronger than the desired upper bound \eqref{eq:IntegralDamn}.
\end{proof}

\subsection{Step V}\label{ssec:StepV}
We conclude the proof of Theorem \ref{Theorem:ShortSum}. Use Lemma \ref{Lemma:TheIntegral} to see that if
the GRH is true, $L(s,\rho)$ is entire, $x > e^2$, $d = n_{\K} - 1$, $c = 1/\ln x $, and $h = x^{1-\mu}$
for some $0 < \mu \leq \frac{1}{2} - \frac{2\ln\ln x}{\ln x}$ such that $1 < x^{\mu} \leq \sqrt{x}
(\ln{x})^{-2}$, then \eqref{eq:Helmet} and \eqref{eq:IntegralDamn} provide
\begin{align*}
\big|\ln &{L(1,\rho)} - \Sigma(x)\big|                                                                            \\
&\leq d \opn{G}(x,x^{1 - \mu})
      + \opn{J}(x,\mu) \frac{e (\ln{x})^2}{\pi \sqrt{x}}
         \Big(\opn{H}(\tfrac{1}{2},x^{\mu}) \Big(1 + \frac{1}{\mu\ln{x}}\Big)d + \frac{3}{2}\ln|\Delta_{\K}|\Big) \\
&\leq \Bigg(\frac{\opn{F}_1(x,\mu)}{x^{\mu} \ln{x}} + \frac{\opn{F}_2(x,\mu)(\ln{x})^3}{\sqrt{x}}\Bigg)(n_{\K} - 1)
      + \frac{\opn{F}_3(x,\mu) (\ln{x})^2}{\sqrt{x}} \ln|\Delta_{\K}| ,
\end{align*}
in which
\begin{align}
\opn{F}_1(x,\mu)
&= 1 + \frac{3 x^{\mu - \frac{1}{2}} (\ln x + x^{-\mu} + 2) \ln{x}}{4\pi} ,
\label{eq:A1} \\
\opn{F}_2(x,\mu)
&= \frac{3 e}{\pi \ln{x}} \Bigg(\ln \bigg( \frac{\pi ^{3/2} \zeta (\frac{3}{2})}{6 \sqrt{2}} \bigg) + \frac{\ln(\tfrac{3}{2} + \sqrt{9+x^{2\mu}})}{2}\Bigg) \label{eq:A2} \\[5pt]
&\qquad \times\Big(\frac{1}{2} + \frac{2}{\ln{x} - 2} + \frac{2 \big(1 + x^{-\mu}\big)^{\frac{1}{2} + \frac{1}{\ln{x}}}}{\ln{x}}\Big)
\Big(1 + \frac{1}{\mu\ln{x}}\Big) ,
\quad\text{and}\nonumber \\[5pt]
\opn{F}_3(x,\mu)
&= \frac{3 e}{4 \pi} \Big(1 + \frac{4}{\ln{x} - 2} + \frac{4 \big(1 + x^{-\mu}\big)^{\frac{1}{2} + \frac{1}{\ln{x}}}}{\ln{x}}\Big) .
\label{eq:A3}
\end{align}
Note that $x^{\mu} \leq \sqrt{x} (\ln{x})^{-2}$ ensures that each $\opn{F}_i(x,\mu)$ is a bounded
function of $x$.
Now, let $\mu = \frac{1}{2}-2 \frac{\ln \ln x}{\ln x}$, so that $x^{\mu} = \sqrt{x} (\ln{x})^{-2}$.
It follows that as $x\to\infty$, we have
\begin{equation*}
\mu \to \frac{1}{2}, \quad
\opn{F}_1(x,\mu) \to 1 + \frac{3}{4\pi} , \quad
\opn{F}_2(x,\mu) \to \frac{3e}{8\pi} , \quad\text{and}\quad
\opn{F}_3(x,\mu) \to \frac{3e}{4\pi} .
\end{equation*}
Moreover, we observe that on $x \geq 5\cdot 10^5$, $x^{\mu}$ is increasing,
\begin{align*}
&\opn{F}_1(x,\mu) - 1 - \frac{3}{4\pi}
\leq \frac{0.54}{\ln{x}} ,
\qquad
\opn{F}_2(x,\mu) - \frac{3e}{8\pi}
\leq \frac{1.35}{\ln{x}} , \\
&\qquad\qquad\text{and}\quad
\opn{F}_3(x,\mu) - \frac{3e}{4\pi}
\leq \frac{6.01}{\ln{x}} .
\end{align*}
%
%
%
Combine these observations with $L(1,\rho) = \kappa_{\K}$ to see that
\begin{multline*}
\big|\ln{\kappa_{\K}} -  \Sigma(x)\big|
\leq \Big(\frac{3e}{8\pi} + \frac{1.35}{\ln{x}} + \frac{1 + \frac{3}{4\pi}}{(\ln{x})^2} + \frac{0.54}{(\ln{x})^3}\Big) \frac{(\ln{x})^3}{\sqrt{x}} (n_{\K} - 1) \\
   + \Big(\frac{3e}{4\pi} + \frac{6.01}{\ln{x}}\Big) \frac{(\ln{x})^2}{\sqrt{x}} \ln|\Delta_{\K}| ,
\end{multline*}
which is marginally stronger than the desired upper bound. \qed

\section{Proof of Theorem \ref{Theorem:Main}}\label{Section:Proof}

Suppose in this section that $\K \neq \Q$ is a number field, GRH for $\zeta_\K$ is true, and
$\zeta_{\K}/\zeta$ is entire. Here we prove Theorem \ref{Theorem:Main}. To this end, we introduce key
definitions and some important parameter choices in Section \ref{Subsection:Setup}. We then prove
\eqref{eq:BoundUpper} and \eqref{eq:BoundLower_ii} in Subsections \ref{Subsection:Upper} and
\ref{Subsection:Lower}, respectively.

\subsection{Setup}\label{Subsection:Setup}
The function
\begin{equation}\label{eq:M}
\opn{M}(t) = 1 + \frac{4\ln \ln \ln t}{\ln\ln t}
\end{equation}
has a maximum $1 + \frac{4}{e} \approx 2.4715\ldots$ at $t = \exp(e^e)$. Let
\begin{equation}\label{eqn:x-small-ell}
x = (\ln |\Delta_{\K}|)^{2\opn{M}(| \Delta_{\K} |)}
  = (\ln |\Delta_{\K}|)^2 (\ln \ln |\Delta_{\K}| )^{8}
\end{equation}
and observe that $x \geq 5\cdot 10^5$ on $|\Delta_{\K}| \geq 1.6\cdot 10^6$,
%
so Lemma \ref{lem:SigmaUpper}, Lemma \ref{lem:SigmaLower}, and Theorem \ref{Theorem:ShortSum} apply
freely in what follows. Note for reference that $\ln x = 2 \opn{M}(|\Delta_{\K}|) \ln \ln |\Delta_{\K}|$.
\begin{remark}
The choice \eqref{eqn:x-small-ell} is necessary to ensure that the error terms of asymptotic order
$x^{-1/2}(\ln{x})^3 (n_{\K} - 1)$ and $x^{-1/2}(\ln{x})^2\ln{|\Delta_{\K}|}$ in our arguments are
decreasing as $|\Delta_{\K}| \to \infty$ with the right shape.
\end{remark}

\subsection{Upper bound}\label{Subsection:Upper}
In this subsection we prove the upper bound \eqref{eq:BoundUpper}. To begin, observe that Theorem
\ref{Theorem:ShortSum} implies that if $x \geq 5\cdot 10^5$, then
\begin{align*}
\kappa_{\K}
&\leq e^{\Sigma(x)
         + \big(\frac{3e}{8\pi} + \frac{1.45}{\ln{x}}\big) \frac{(\ln{x})^3}{\sqrt{x}} (n_{\K} - 1)
         + \big(\frac{3e}{4\pi} + \frac{6.01}{\ln{x}}\big) \frac{(\ln{x})^2}{\sqrt{x}} \ln|\Delta_{\K}|} \\
&\leq e^{\Sigma(x) + 0.435\frac{(\ln{x})^3}{\sqrt{x}} (n_{\K} - 1) + 1.11\frac{(\ln{x})^2}{\sqrt{x}} \ln|\Delta_{\K}|} .
\end{align*}
%
Furthermore, Lemma \ref{lem:SigmaUpper} implies that if $x \geq 5\cdot 10^5$, then
\begin{align*}
\kappa_{\K}
&\leq e^{(n_{\K} - 1) \big(\ln\ln{x} + \gamma + \frac{3(\ln x  + 2)}{8\pi \sqrt{x}}\big)
      + 0.435\frac{(\ln{x})^3}{\sqrt{x}} (n_{\K} - 1)
      + 1.11\frac{(\ln{x})^2}{\sqrt{x}} \ln|\Delta_{\K}|} \\
&\leq (e^{\gamma}\ln{x})^{n_{\K} - 1}
   e^{0.44 \frac{(\ln{x})^3 (n_{\K} - 1)}{\sqrt{x}}
    + 1.11 \frac{(\ln{x})^2}{\sqrt{x}} \ln|\Delta_{\K}|} .
\end{align*}
%
Thus, \eqref{eqn:x-small-ell} ensures that for $|\Delta_{\K}| \geq 1.6\cdot 10^6$,
\begin{align*}
\kappa_{\K}
&\leq (2 \opn{M}(|\Delta_{\K}|) e^{\gamma}\ln\ln{|\Delta_{\K}|})^{n_{\K} - 1}
      e^{\frac{0.44 (2 \opn{M}(|\Delta_{\K}|))^3 (n_{\K} - 1)}{\ln{|\Delta_{\K}|} \ln\ln{|\Delta_{\K}|}}
       +\frac{1.11 (2 \opn{M}(|\Delta_{\K}|))^2}{(\ln\ln{|\Delta_{\K}|})^2}} \\
&\leq \Big(2 \opn{M}(|\Delta_{\K}|)
      e^{\gamma + \frac{0.44 (2 \opn{M}(|\Delta_{\K}|))^3}{\ln{|\Delta_{\K}|} \ln\ln{|\Delta_{\K}|}}
                + \frac{1.11 (2 \opn{M}(|\Delta_{\K}|))^2}{(\ln\ln{|\Delta_{\K}|})^2}}
      \ln\ln{|\Delta_{\K}|}\Big)^{n_{\K} - 1}
\end{align*}
We also have $1 + w \leq e^{w}$ for every $w > 0$. Therefore, if $|\Delta_{\K}| \geq 1.6\cdot 10^6$, then
\begin{align*}
\kappa_{\K}
&\leq \Big(2
      e^{\gamma + \frac{4\ln\ln\ln{|\Delta_{\K}|}}{\ln\ln{|\Delta_{\K}|}} + \frac{0.44 (2 \opn{M}(|\Delta_{\K}|))^3}{\ln{|\Delta_{\K}|} \ln\ln{|\Delta_{\K}|}}
                + \frac{1.11 (2 \opn{M}(|\Delta_{\K}|))^2}{(\ln\ln{|\Delta_{\K}|})^2}}
      \ln\ln{|\Delta_{\K}|}\Big)^{n_{\K} - 1} \\
&\leq \Big(2
      e^{\gamma + \frac{18.3\ln\ln\ln{|\Delta_{\K}|}}{\ln\ln{|\Delta_{\K}|}}}
      \ln\ln{|\Delta_{\K}|}\Big)^{n_{\K} - 1}
= \Big(2
      e^{\gamma}
      (\ln\ln{|\Delta_{\K}|})^{1 + \frac{18.3}{\ln\ln{|\Delta_{\K}|}}}\Big)^{n_{\K} - 1}.
\end{align*}
%
This completes the proof of \eqref{eq:BoundUpper} for all $|\Delta_{\K}| \geq 1.6\cdot 10^6$. The claim
for $|\Delta_\K|< 1.6\cdot 10^6$ is verified directly, since all fields in this range are explicitly
tabulated (Remark \ref{Remark:End}).

\subsection{Lower bound}\label{Subsection:Lower}
In this subsection we prove the lower bound \eqref{eq:BoundLower_ii}. As above, observe that Theorem
\ref{Theorem:ShortSum} implies that if $x \geq 5\cdot 10^5$, then
\begin{align*}
\kappa_{\K}
&\geq e^{\Sigma(x) - 0.435\frac{(\ln{x})^3}{\sqrt{x}} (n_{\K} - 1)
         - 1.11\frac{(\ln{x})^2}{\sqrt{x}} \ln|\Delta_{\K}|} .
\end{align*}
Next, Lemma \ref{lem:SigmaLower} implies that if $x \geq 5\cdot 10^5$, then
\begin{align*}
\kappa_{\K}
&\geq \frac{\zeta(n_{\K})}{e^{\gamma}\ln{x}} \Big(1 - \frac{3\ln{x}}{8\pi\sqrt{x}}\Big)
      e^{- \frac{n_{\K} - 1}{\ln{x}} \big( 1 + \frac{(\ln{x})^2}{8\pi\sqrt{x}} + \frac{1.3\ln{x}}{x} + 0.435\frac{(\ln{x})^4}{\sqrt{x}} \big)
         - 1.11\frac{(\ln{x})^2}{\sqrt{x}} \ln|\Delta_{\K}|} \\
&\geq \frac{\zeta(n_{\K})}{e^{\gamma}\ln{x}} \Big(1 - \frac{3\ln{x}}{8\pi\sqrt{x}}\Big)
      e^{- \frac{n_{\K} - 1}{\ln{x}} \big(1 + 0.44\frac{(\ln{x})^4}{\sqrt{x}} \big) - 1.11\frac{(\ln{x})^2 \ln|\Delta_{\K}|}{\sqrt{x}}} .
\end{align*}
%
Now, we have $1 - w \geq e^{-2w}$ for $w\leq 1/2$. The assumption $x\geq 5\cdot 10^5$ ensures that
$\frac{3\ln{x}}{8\pi\sqrt{x}} \leq \frac{1}{2}$
%
and we get that
\begin{align*}
\kappa_{\K}
&\geq \frac{\zeta(n_{\K})}{e^{\gamma}\ln{x}}
      e^{- \frac{n_{\K} - 1}{\ln{x}} \big(1 + 0.44\frac{(\ln{x})^4}{\sqrt{x}} + 1.11 \frac{(\ln{x})^3 \ln|\Delta_{\K}|}{\sqrt{x}} + \frac{3(\ln{x})^2}{4\pi\sqrt{x}} \big)} .
\end{align*}
By \eqref{eqn:x-small-ell}, for $|\Delta_\K|\geq 1.6\cdot 10^6$ we also have
\begin{align*}
    \frac{0.44(\ln{x})^4}{\sqrt{x}}
    &= \frac{0.44 (2 \opn{M}(| \Delta_{\K} |))^4}{\ln|\Delta_{\K}|} \leq 18.38 , \\
    \frac{1.11 (\ln{x})^3 \ln|\Delta_{\K}|}{\sqrt{x}}
    &= \frac{1.11 (2 \opn{M}(| \Delta_{\K} |))^3}{\ln\ln|\Delta_{\K}|} \leq 50.4 , \quad\text{and}\\
    \frac{3(\ln{x})^2}{4\pi\sqrt{x}}
    &= \frac{3 (2\opn{M}(| \Delta_{\K} |))^2}{4\pi\ln|\Delta_{\K}| (\ln\ln|\Delta_{\K}|)^2}
    \leq 0.06 .
\end{align*}
%
Note that $1 + 18.38 + 50.4 + 0.06 = 69.84$.  For $|\Delta_{\K}| \geq 1.6\cdot 10^6$, we have
\begin{align*}
\kappa_{\K}
&\geq \frac{\zeta(n_{\K})}{2 \opn{M}(| \Delta_{\K} |) e^{\gamma}\ln\ln|\Delta_{\K}|}
      e^{- \frac{69.84 (n_{\K} - 1)}{2 \opn{M}(|\Delta_{\K}|) \ln\ln|\Delta_{\K}|}} \\
\intertext{and since $(1 + w)^{-1} \geq e^{-w}$ for every $w > 0$, we get}
\kappa_{\K}
&\geq \frac{\zeta(n_{\K})}{2 e^{\gamma}\ln\ln|\Delta_{\K}|}
      e^{- \frac{(n_{\K} - 1)\ln\ln\ln|\Delta_{\K}|}{\ln\ln|\Delta_{\K}|} \left(4 + \frac{69.84}{2\opn{M}(|\Delta_{\K}|) \ln\ln\ln|\Delta_{\K}|} \right)} \\
&\geq \frac{\zeta(n_{\K})}{2 e^{\gamma}\ln\ln|\Delta_{\K}|}
      e^{- \frac{18.5 (n_{\K} - 1)\ln\ln\ln|\Delta_{\K}|}{\ln\ln|\Delta_{\K}|} }
= \frac{\zeta(n_{\K})}{2 e^{\gamma} (\ln\ln|\Delta_{\K}|)^{1 + \frac{18.5 (n_{\K} - 1)}{\ln\ln|\Delta_{\K}|}}} .
\end{align*}

This completes the proof of \eqref{eq:BoundLower_ii} for all $|\Delta_{\K}| \geq 1.6\cdot 10^6$. The
claim for $|\Delta_\K|< 1.6\cdot 10^6$ is verified directly, since all fields in this range are
explicitly tabulated.
\begin{remark}\label{Remark:End}
Every field with $|\Delta_{\K}| \leq 1.6\cdot 10^6$ has degree $\leq 8$. In
\url{https://sites.unimi.it/molteni/research/bounds_residue/bounds_residue.html}, we have made available
the list of those fields, as downloaded from \texttt{LMFDB}~\cite{LMFDB}, the script we used to check the
theorem for them, and the results we obtained for each of them (we computed for each field a constant
that can replace $19$ in~\eqref{eq:BoundUpper} and a similar one for~\eqref{eq:BoundLower_ii}). An
analysis of these results shows that for fields with $|\Delta_{\K}|\leq 1.6\cdot 10^6$, the bounds hold
true also in the version where both occurrences of $19$ are replaced with $0$, except for
$\Q\big(\sqrt{-163}\big)$ and those with $-4\leq \Delta_{\K}\leq 8$.
\end{remark}


\end{document}